\documentclass[a4paper,10pt]{amsart}
\usepackage{amsmath,amsfonts,amssymb,amsthm, amsfonts, amscd, url}
\usepackage[latin1]{inputenc}
\usepackage[english]{babel}
\usepackage{url}

\usepackage[cmtip,arrow]{xy}
\usepackage{pb-diagram,pb-xy}

\swapnumbers 
\theoremstyle{plain}
\newtheorem{thm}{Theorem}[section]
\newtheorem{prop}[thm]{Proposition}
\newtheorem{lemma}[thm]{Lemma}

\theoremstyle{remark}
\theoremstyle{definition}
\newtheorem{rem}[thm]{Remark}
\newtheorem{rems}[thm]{Remarks}
\newtheorem{remdef}[thm]{Remark-Definition}
\newtheorem{remsdefs}[thm]{Remarks-Definitions}

\newtheorem{defi}[thm]{Definition}

\newtheorem{recall}[thm]{Recall}

\newtheorem{notas}[thm]{Notations}

\usepackage[english]{babel}

\title[Real square roots of matrices]{Real square roots of matrices:\\ differential properties in semi-simple, symmetric and orthogonal cases}

\author{Alberto Dolcetti \and Donato Pertici}

\begin{document}

\parindent 0pt
\selectlanguage{english}

\maketitle

\vspace*{-0.2in}

\begin{center}
{\scriptsize Dipartimento di Matematica  e Informatica, Viale Morgagni 67/a, 50134 Firenze, ITALIA

\vspace*{0.07in}

alberto.dolcetti@unifi.it, \  \url{http://orcid.org/0000-0001-9791-8122}

\vspace*{-0.03in}

donato.pertici@unifi.it,  \   \url{http://orcid.org/0000-0003-4667-9568}}

\end{center}



\begin{abstract}
We study the differential and metric structures of the set of real square roots of a non-singular real matrix, under the assumption  that the matrix and its square roots are semi-simple, or symmetric, or orthogonal.
\end{abstract}


{\small \tableofcontents}

\renewcommand{\thefootnote}{\fnsymbol{footnote}}
\footnotetext{
This research was partially supported by GNSAGA-INdAM (Italy).
}
\renewcommand{\thefootnote}{\arabic{footnote}}
\setcounter{footnote}{0}

\vspace*{-0.3in}

{\small Keywords: square root matrix, semi-simple matrix, symmetric matrix, orthogonal matrix, homogeneous space, trace metric, totally geodesic semi-Riemannian submanifold.}

\smallskip

{\small Mathematics~Subject~Classification~(2020): 15A24, 53C30, 15B10.}

\bigskip

\section*{Introduction}\label{intro}
A square root of an $n\times n$ matrix $M$ is any matrix $Y$ such that $Y^2 = M$.

As remarked in \cite{Hi2020}, matrix square roots arise in many applications, often in connection with other matrix problems such as polar decomposition, matrix geometric means, Markov chains, quadratic matrix equations, generalized eigenvalue problems and recently in connection with machine learning.
For more information we refer to \cite{HoJ1991}, \cite{Hi2008}, \cite{HoJ2013} and \cite{Hi2020}.

As suggested in \cite{Hi2020}, if $M$ satisfies certain given properties, it is natural to ask if there exist square roots, having the same or analogous properties. 

Moreover, if the set of such square roots of $M$ is not empty, it seems also natural to us to ask what sort of algebraic, topological, differential or metric structures can be induced by $M$ on this set.

In this paper we work with semi-simple matrices in $GL_n$ (the set of real $n \times n$  non-singular matrices).
It is well-known that such matrices admit real square roots if and only if their (possible) negative eigenvalues have even multiplicity: for this reason in this paper we work also under this further assumption.

Let $M \in GL_n$ be semi-simple, whose (possible) negative eigenvalues have even multiplicity;
our aim is to study the following sets:

- $\mathcal{SR}(M)$, the set of all real square roots of $M$ (see \S \ref{SR});

- $\mathcal{SSR}(M)$, the set of all real symmetric square roots of $M$, when $M$ is supposed to be symmetric positive definite (see \S \ref{SSR});

- $\mathcal{OSR}(M)$, the set of all orthogonal square roots of $M$, when $M$ is supposed to be special orthogonal (see \S \ref{OSR}).

Preliminary material is in \S\,\ref{prelim}, where, in particular, we recall the so-called \emph{trace metric}, $g$, and its properties on the manifold of real non-singular matrices and on its submanifolds of symmetric matrices with fixed signature and of orthogonal matrices (Recall \ref{trace_metric}). We have already considered such metric in \cite{DoPe2015}, \cite{DoPe2018a}, \cite{DoPe2019}, \cite{DoPe2020} and now we adapt our previous techniques and results to the study of square roots.

The trace metric (called also \emph{affine-invariant metric} or \emph{Fisher-Rao metric}) is widely considered in the setting of positive definite real matrices (together with some other metrics) in order to deal with many applications. For more information we refer, for instance and among many others, to \cite{BridHaef1999}, \\ \cite{BhaH2006}, \cite{Bha2007}, \cite{MoZ2011}, \cite{Bha2013}, \\ \cite{Amari2016}, \cite{Terras2016}, 
\cite{BarNie2017}), \cite{NieBar2019}).

In general the sets $\mathcal{SR}(M)$, $\mathcal{SSR}(M)$ and $\mathcal{OSR}(M)$ have not a structure of manifold, but each of them is disjoint union of suitable subsets, consisting in matrices having the same eigenvalues with the same multiplicities (Remarks-Definitions \ref{Jordan-form}, \ref{Jordan-form-SSR} and \ref{rem-RJS-RJA-ortog}). All these subsets are manifolds of various dimensions; we describe them explicitly and investigate their differential and metric properties (\S \ref{SR}, \S \ref{SSR} and \S \ref{OSR} respectively).

In \S\,\ref{SR} we prove that each of these suitable subsets of $\mathcal{SR}(M)$ is a closed embedded totally geodesic homogeneous semi-Riemannian submanifold of $(GL_n, g)$ (Proposition \ref{descr-SR(M)-indici} and Theorem \ref{SR(M)-metrica}).

In \S\,\ref{SSR} we prove that the above suitable subsets of $\mathcal{SSR}(M)$ agree with its connected components and that they are compact totally geodesic homogeneous semi-Riemannian submanifolds of $(GL_n, g)$, always diffeomorphic to the product of real Grassmannians (Proposition \ref{descr-SSR(M)-indici} and Theorem \ref{SSR(M)-metrica}).

Finally, in \S\,\ref{OSR} we prove that each of the previous suitable subsets of $\mathcal{OSR}(M)$ is a compact totally geodesic homogeneous Riemannian submanifold of the Riemannian manifold of orthogonal matrices endowed with the metric induced by the \emph{Frobenius metric} (Proposition \ref{prop-ortog} and Theorem \ref{manifold-O-con-indici}). Indeed this last metric agrees on $\mathcal{O}_n$ with the opposite of the trace metric.

\section{Preliminary facts}\label{prelim}

\begin{notas}\label{notazioni}\ \\
$M_n$ (and $Sym_n$): the $\mathbb{R}$-vector space of the real square matrices of order $n$ (which are symmetric);

$GL_n$: the multiplicative group of the non-singular real matrices of order $n$;

$GLSym_n(q)$: the matrices of $GL_n \cap Sym_n$ having $q$ positive eigenvalues and $n-q$ negative eigenvalues (i.e. having \emph{signature} $(q, n-q)$);

$\mathcal{O}_n$ (and $S\mathcal{O}_n$ or $\mathcal{O}_n^-$): the multiplicative group of real \emph{orthogonal} matrices of order $n$ (with determinant $1$ or $-1$);

$\mathfrak{so}_n$: the Lie algebra of skew-symmetric real matrices of order $n$;

$M_n(\mathbb{C})$: the $\mathbb{C}$-vector space of the complex square matrices of order $n$;

$GL_n(\mathbb{C})$: the multiplicative group of the non-singular complex matrices of order $n$;

$Herm_n$: the vector space of the complex \emph{hermitian} matrices of order $n$;

$Herm_n(\mu)$: the matrices of $GL_n(\mathbb{C}) \cap Herm_n$ having $\mu$ positive eigenvalues and $n - \mu$ negative eigenvalues (i.e. having signature $(\mu, n - \mu$));

$U_n$: the multiplicative group of complex \emph{unitary} matrices of order $n$; 

$I_n$: the identity matrix of order $n$;

${\bf i}$: the \emph{imaginary unit}.

\smallskip

We write $\bigsqcup_j X_j$ to emphasize the union of mutually disjoint sets.

\smallskip

For every $A \in M_n(\mathbb{C})$, $tr(A)$ is its \emph{trace}, $A^T$ is its \emph{transpose}, $A^*:=\overline{A\,}^T$ is its \emph{transpose conjugate}, $det(A)$ is its \emph{determinant} and, provided that $det(A) \ne 0$, $A^{-1}$ is its \emph{inverse}.

For every $\theta \in \mathbb{R}$, we denote
 $E_{\theta}:=\begin{pmatrix} 
 \cos \theta & - \sin \theta \\ 
\sin \theta &  \cos \theta
\end{pmatrix}$; 
hence $E_\theta = (\cos \theta)I_2 + (\sin \theta) E_{\pi/2}$.

\smallskip

If $B_1, \cdots , B_m$ are square matrices (of various orders), $B_1 \oplus \cdots \oplus B_m$ is the block diagonal square matrix with $B_1, \cdots , B_m$ on its diagonal and, for every square matrix $B$, $B^{\oplus m}$ denotes $B \oplus \dots \oplus B$ ($m$ times). The notations $(\pm I_0) \oplus B$ and $B \oplus (\pm I_0)$ simply indicate the matrix $B$.

If $\mathcal{S}_1, \dots , \mathcal{S}_m$ are sets of square matrices, then  $\mathcal{S}_1 \oplus \dots \oplus \mathcal{S}_m$ denotes the set of all matrices $B_1 \oplus \cdots \oplus B_m$ with $B_j \in \mathcal{S}_j$ for every $j$.

\smallskip
For every matrix $B \in GL_n$ we denote
$\mathcal{C}_B := \{X \in GL_n : [B, X] = 0\}$, where $[B,X] := BX-XB$ is the usual \emph{commutator} of $B$ and $X$.

It is easy to prove that $\mathcal{C}_B$ is a closed Lie subgroup of $GL_n$.

\smallskip

For any other notation and for information on the matrices, not explicitly recalled here, we refer to \cite{HoJ2013} and to \cite{Hi2008}.
\end{notas}

\begin{defi}
Let $M \in M_n$ any matrix. A \emph{real square root} of $M$ is every matrix of $M_n$, solving the matrix equation $X^2=M$.
\end{defi}

\begin{rem}\label{esist-rad-quadr}
When $M$ is non-singular, the following fact is well-known (see for instance \cite[Thm.\,1.23]{Hi2008}): 

$M \in GL_n$ has a real square root if and only if it has an even number of Jordan blocks of each size, for every 
negative eigenvalue.

Note that, if the matrix $M \in GL_n$ is semi-simple too, then it has a real square root if and only if each of its negative eigenvalues has even multiplicity; hence no matrix in $\mathcal{O}_n^-$ has real square roots.
\end{rem}

\begin{notas}\label{def-XSR}
Assume that $M \in GL_n$ is semi-simple and that its (possible) negative eigenvalues have even multiplicity. 

We want to study the following sets:

\smallskip

$\mathcal{SR}(M)$, the set of all real square roots of $M$ (see \S \ref{SR});

\smallskip

$\mathcal{SSR}(M):=\mathcal{SR}(M) \cap Sym_n$, the set of all real symmetric square roots of $M$, when $M$ is supposed to be symmetric positive definite (see \S \ref{SSR});

\smallskip

$\mathcal{OSR}(M) :=\mathcal{SR}(M) \cap \mathcal{O}_n$, the set of all orthogonal square roots of $M$, when $M$ is supposed to be an element of $S\mathcal{O}_n$ (see \S \ref{OSR}).
\end{notas}

\begin{recall}\label{trace_metric}
We can consider the so-called \emph{trace metric} on $GL_n$ , given by
$$g_A(V,W) = tr(A^{-1}V A^{-1}W)$$

for every $A \in GL_n$ and every $V, W \in T_A(GL_n) = M_n$.

\smallskip

The trace metric, $g$, induces a semi-Riemannian structure on $GL_n$ and its restrictions (denoted always with the same symbol $g$) induce a semi-Riemannian structure on $GLSym_n(q)$, for every $q$, a Riemannian structure on the manifold, $GLSym_n(n)$, of real symmetric positive definite matrices, while, on $\mathcal{O}_n$, the opposite metric, $-g$, agrees with the Riemannian metric induced by the usual (flat) \emph{Frobenius metric} of $M_n$.
Moreover $(\mathcal{O}_n , g)$ and all $GLSym_n(q)$ are totally geodesic semi-Riemannian submanifolds  of $(GL_n, g)$.

As recalled in Introduction, we have studied all these metrics and their properties in \cite{DoPe2015}, \cite{DoPe2018a}, \cite{DoPe2019}, \\ \cite{DoPe2020}, to which we refer for more information.
\end{recall}

\begin{remsdefs}\label{def-isom}

a) In this paper we consider the following maps, defined on $GL_n$:

- $j(X) := X^{-1}$ (\emph{inversion});

- $L_M (X) := MX$ (\emph{left-translation} by a fixed matrix $M \in GL_n$);

- $R_M(X):= XM$ (\emph{right-translation} by a fixed matrix $M \in GL_n$);

- $\widehat{\Gamma}_M(X): = MXM^{-1}$ (\emph{conjugation} by a fixed matrix $M \in GL_n$);

- $\Gamma_M(X): = MXM^T$ (\emph{congruence} by a fixed matrix $M \in GL_n$).

\smallskip

The above maps are isometries of $(GL_n, g)$, for any $M \in GL_n$, while they are isometries also of $(\mathcal{O}_n, \pm g)$, for any $M \in \mathcal{O}_n$. Finally the maps $j$ and $\Gamma_M$ are isometries of $(GLSym_n(q), g)$, for any $M \in GL_n$ 
(see again \cite{DoPe2015}), \cite{DoPe2018a}, \cite{DoPe2019}, \cite{DoPe2020} for more information).

\smallskip 

b) For every bijection $F$ of $GL_n$ (or of $\mathcal{O}_n$) the set of \emph{fixed points} of $F$ in $GL_n$ (or in $\mathcal{O}_n$) is denoted by $Fix(F)$.

\smallskip

For any matrix $M \in GL_n$, the set $Fix(L_M \circ j)$ agrees with the sets of square roots of $M$: indeed $Y \in Fix(L_M \circ j)$ if and only if $Y = MY^{-1}$ if and only if $Y^2 = M$.

If $M \in S\mathcal{O}_n$, $L_M \circ j$ defines also a map from $\mathcal{O}_n$ onto itself and, so, $Fix(L_M \circ j) = \mathcal{OSR}(M)$.
\end{remsdefs}

\begin{rem}\label{Popov}
Let $G$ be a real Lie group acting smoothly on a differentiable manifold $X$. The orbit of every $x \in X$ is an immersed submanifold of $X$, diffeomorphic to the homogeneous space $\dfrac{G}{G_x}$, where $G_x$ is the isotropy subgroup of $G$ at $x$. 

This submanifold is not necessarily embedded in $X$, but, if $G$ is compact, then all orbits are embedded submanifolds (see \cite[p.\,1]{Popov1991}). 
\end{rem}

\begin{remdef}\label{rho}
The mapping 
$\rho_1: \mathbb{C} \to M_2$, $\rho_1(z) =  Re(z) I_2 + Im(z) E_{\pi/2}$,
 is clearly a monomorphism of $\mathbb{R}$-algebras. 
 
Note that $\rho_1({\overline{z}}) = \rho_1(z)^T$
and that $\rho_1(z) \in GL_2$ as soon as $z \ne 0$. 

More generally, for any $h \ge 1$, we define the mapping $\rho_h: M_h(\mathbb{C}) \to M_{2h}$, which maps the $h \times h$ complex matrix $Z=(z_{ij})$ to the $(2h) \times (2h)$ block real matrix  $(\rho_1(z_{ij}))$, having $h^2$ blocks of order $2 \times 2$. 

We have already considered the mappings $\rho_h$ in \cite[Rem.\,2.8]{DoPe2020}; in particular we  recall that
$tr(\rho_h(Z)) = 2 Re(tr(Z)))$, $det(\rho_h(Z)) = |det(Z)|^2$ and that $\rho_h$ is a monomorphism of $\mathbb{R}$-algebras, whose restriction to $GL_h(\mathbb{C})$ has image into $GL_{2h}$ and it is a monomorphism of Lie groups.

Furthermore we have: $\rho_h(Z^*) = \rho_h(Z)^T$, so the restriction of $\rho_h$ to $U_h$ is a mono\-morphism of Lie groups
and $\rho_h (U_h)=  \rho_h (GL_h(\mathbb{C})) \cap S\mathcal{O}_{2h} = \rho_h (GL_h(\mathbb{C})) \cap \mathcal{O}_{2h}$; analogously the restriction of $\rho_h$ to $Herm_h$ has image into $Sym_{2h}$ and it is  a monomorphism of $\mathbb{R}$-vector spaces. 

To simplify the notations and in absence of ambiguity, we will omit to write the symbol $\rho_h$, so, for instance, we can consider $M_h(\mathbb{C})$ as an $\mathbb{R}$-subalgebra of $M_{2h}$, $GL_h(\mathbb{C})$ as a Lie subgroup of $GL_{2h}$,  $U_h$ as a Lie subgroup of $S\mathcal{O}_{2h}$ and $Herm_h$ as an $\mathbb{R}$-vector subspace of $Sym_{2h}$.
\end{remdef}

\section{Semi-simple case}\label{SR}

Aim of this section is to study $\mathcal{SR}(M)$: the set of all real square roots of a non-singular semi-simple matrix $M$, whose (possible) negative eigenvalues have even multiplicity (Notations \ref{def-XSR}).

\begin{rem}\label{AA^2semisemplici}
Let $A \in GL_n$. Then $A$ is semisimple if and only if $A^2$ is semisimple.

\smallskip

One implication is trivial.

For the other, assume that $A^2$ is semisimple. The multiplicative Jordan-Chevalley decomposition implies that there is a semisimple matrix $S$ and a nilpotent matrix $N$ such that $A= S(I+N)=(I+N)S$ (see for instance \cite[\S\,4.2]{Humph1975} and also \cite{DoPe2017}, \cite{DoPe2018b}). 

Now $A^2=S^2(I+2N+N^2)$ with $S^2$ semisimple and $2N+N^2$ nilpotent, so from the uniqueness of the multiplicative Jordan-Chevalley decomposition we get $A^2=S^2$ and $N^2 + 2N = 0$. Hence the minimal polynomial, $x^k$, of $N$ is a divisor of $x^2+2x$, hence $k=1$, $N=0$ and $A$ is semisimple.
\end{rem}

\begin{remsdefs}\label{Jordan-form}

a) Let $M \in GL_n$ be a semisimple matrix, whose every (possible) negative eigenvalue has even multiplicity, and denote its eigenvalues in the following way: 

\smallskip

- the distinct positive eigenvalues are: $\lambda_1 < \lambda_2 < \cdots < \lambda_p$ ($p \ge 0$) with multiplicity $h_1, h_2, \cdots , h_p$ respectively (if $p \ge 1$);

\smallskip

- the distinct non-real eigenvalues are: $\rho_{1,1}\exp(\pm {\bf i} \theta_1), \cdots , \rho_{1,s_1}\exp(\pm {\bf i} \theta_{1})$, 

$\rho_{2,1}\exp(\pm {\bf i} \theta_2), \cdots , \rho_{2,s_2}\exp(\pm {\bf i} \theta_{2})$ up to 
$\rho_{r,1}\exp(\pm {\bf i} \theta_r), \cdots , \rho_{r,s_r}\exp(\pm {\bf i} \theta_{r})$, 

where both $\rho_{lt} \exp(\pm {\bf i} \theta_{l})$ have multiplicity $m_{lt}$ for every admissible $l$ and $t$ ($r \ge 0$) and  where $0 < \theta_{1} < \theta_2 < \cdots < \theta_r < \pi$ and $0 < \rho_{i1} < \rho_{i2} < \cdots  < \rho_{is_i}$ for every $i = 1, \cdots , r$ (if $r \ge 1$);

\smallskip

- the distinct negative eigenvalues are: $-\zeta_1 > -\zeta_2 > \cdots > -\zeta_q$ ($q \ge 0$) with multiplicity $2 k_1 , 2 k_2, \cdots , 2k_q$ respectively (if $q \ge 1$).

\smallskip

Note that $\sum_{i=1}^p h_i + 2 \sum_{l=1}^r \sum_{t=1}^{s_l} m_{lt} + 2\sum_{i=1}^q k_i =n$.

\smallskip

We denote by $\mathcal{J}_M$ the matrix
\begin{center}
$\mathcal{J}_M:= \big( \bigoplus_{i=1}^p \lambda_i I_{h_i}\big) \oplus \big(\bigoplus_{l=1}^r \bigoplus_{t=1}^{s_l} \rho_{lt} E_{\theta_l}^{\oplus m_{lt}}\big) \oplus \big( \bigoplus_{i=1}^{q}(- \zeta_i) I_{2 k_i} \big)
$.
\end{center}

Then there is a matrix $C \in GL_n$ such that
$$M= C \mathcal{J}_M C^{-1}$$

(see for instance \cite[Cor.\,3.4.1.10 p.\,203]{HoJ2013}).

\smallskip

Following \cite[Remark-Definition\,1.7]{DoPe2020}, we call such matrix $\mathcal{J}_M$ \emph{the real Jordan standard form} (shortly: RJS form) of $M$. 

\smallskip

b) Let $M$ be as in (a) and denote by $Y \in GL_n$ a square root of $M$. By Remark \ref{AA^2semisemplici}, $Y$ is semisimple and its eigenvalues are (complex) square roots of the eigenvalues of $M$. Hence the eigenvalues of $Y$ are 

\smallskip

- $\sqrt{\lambda_i}$ with multiplicity $u_i \ge 0$ and $- \sqrt{\lambda_i}$ with multiplicity $v_i \ge 0$ where $u_i+v_i = h_i$ ($i = 1, \cdots , p$, if $p \ge 1$);

\smallskip

- $\sqrt{\rho_{lt}} \exp( \pm {\bf i} \theta_l/2)$ both with multiplicity $\mu_{lt}$ and $\sqrt{\rho_{lt}} \exp( \pm {\bf i} (\theta_l/2 -\pi))$ both with multiplicity $\nu_{lt}$, where $\mu_{lt} + \nu_{lt} = m_{lt}$ (for every admissible $l$ and $t$);

\smallskip

- $\pm \sqrt{\zeta_i} {\bf i}$ both with multiplicity $k_i$ ($i = 1 , \cdots , q$, if $q \ge 1$).

\smallskip

We denote by $\widetilde{\mathcal{J}}_Y$ the matrix

$\widetilde{\mathcal{J}}_Y:=$

$[\bigoplus_{i=1}^p (\sqrt{\lambda_i} (I_{u_i} \oplus (-I_{v_i})))]\oplus [\bigoplus_{l=1}^r \bigoplus_{t=1}^{s_l} \sqrt{\rho_{lt}} (E_{\theta_l/2}^{\oplus \mu_{lt}} \oplus E_{(\theta_l/2 -\pi)}^{\oplus \nu_{lt}} )] \oplus [\bigoplus_{i=1}^q \sqrt{\zeta_i} E_{\pi/2}^{\oplus k_i}]$.

Then there exists a matrix $K \in GL_n$ such that
$$
Y=K \widetilde{\mathcal{J}}_Y K^{-1}.
$$
Note that 
$$
E_{\phi-\pi} = - E_{\phi}, \ \forall \phi,  \   \mbox{ and } \
(\widetilde{\mathcal{J}}_Y)^2 = \mathcal{J}_M.
$$  
Remark that the last equality implies that $\mathcal{C}_{\widetilde{\mathcal{J}}_Y} \subseteq \mathcal{C}_{{\mathcal{J}}_M}$.

Analogously to \cite[Remark-Definition\,1.8]{DoPe2020}, we can call the matrix $\widetilde{\mathcal{J}}_Y$ \emph{the real Jordan auxiliary form} (shortly: RJA form) of $Y$.

\smallskip

c) We denote by 
$\mathcal{SR}(M)_{(u_1, \cdots , u_p)}^{(\mu_{11}, \cdots ,\mu_{1s_1}; \cdots ; \mu_{r1}, \cdots ,\mu_{rs_r})}$ the subset of $\mathcal{SR}(M)$ of all real square roots of $M$ whose eigenvalues have the same multiplicities of the eigenvalues of $Y$ as in (b).

\smallskip

Note that any two matrices in $\mathcal{SR}(M)_{(u_1, \cdots , u_p)}^{(\mu_{11}, \cdots ,\mu_{1s_1}; \cdots ; \mu_{r1}, \cdots ,\mu_{rs_r})}$ have the same RJA form; indeed the matrix $M$ and the two families of indices $(u_1, \cdots , u_p)$ and 

$(\mu_{11}, \cdots ,\mu_{1s_1}; \cdots ; \mu_{r1}, \cdots ,\mu_{rs_r})$ determine also the $v_i$'s, the $\nu_{lt}$'s and the $k_j$'s.

\smallskip

One or both families of indices $(u_1, \cdots , u_p)$ and 
$(\mu_{11}, \cdots ,\mu_{1s_1}; \cdots ; \mu_{r1}, \cdots ,\mu_{rs_r})$ can be empty (indeed $p$ and $r$  can vanish). In this case we can use the same notations, by omitting the family or both families of indices.

\smallskip

Note that, if $p=r=0$ (and so $q \ge 1$), then the RJA form of every square root of $M$ is $\bigoplus_{i=1}^q \sqrt{\zeta_i} E_{\pi/2}^{\oplus k_i}$ and that $\mathcal{SR}(M)$ (without any lower and upper index) agrees with the set of all real square roots of $M$.

\smallskip

Note also that, if at least one between $p$ and $r$ is non-zero, then $\mathcal{SR}(M) = \sqcup \, \mathcal{SR}(M)_{(u_1, \cdots , u_p)}^{(\mu_{11}, \cdots ,\mu_{1s_1}; \cdots ; \mu_{r1}, \cdots ,\mu_{rs_r})}$, where the disjoint union is taken on all indices $0 \le u_i \le h_i$ for every $1\le i \le p$ (if $p \ge 1$) and $0 \le \mu_{lt} \le m_{lt}$ for every $1 \le l \le r$ and for every $1 \le t \le s_l$ (if $r \ge 1$).

\smallskip

In the following statements of this section it is understood that one or both families of the above indices can  be empty.
\end{remsdefs}

\begin{prop}\label{descr-SR(M)-indici}
Let $M \in GL_n$ be a semisimple matrix, whose every (possible) negative eigenvalue has even multiplicity and let $C \in GL_n$ be such that $M= C \mathcal{J}_M C^{-1}$ as in Remarks-Definitions \ref{Jordan-form} (a). 

Fix any set $\mathcal{SR}(M)_{(u_1, \cdots , u_p)}^{(\mu_{11}, \cdots ,\mu_{1s_1}; \cdots ; \mu_{r1}, \cdots ,\mu_{rs_r})}$ as in Remarks-Definitions \ref{Jordan-form} (c) and denote by $\widetilde{\mathcal{J}}$ the RJA form of every $Y \in \mathcal{SR}(M)_{(u_1, \cdots , u_p)}^{(\mu_{11}, \cdots ,\mu_{1s_1}; \cdots ; \mu_{r1}, \cdots ,\mu_{rs_r})}$. Then

\smallskip

$\mathcal{SR}(M)_{(u_1, \cdots , u_p)}^{(\mu_{11}, \cdots ,\mu_{1s_1}; \cdots ; \mu_{r1}, \cdots ,\mu_{rs_r})} = \{C X \widetilde{\mathcal{J}} X^{-1} C^{-1} : X \in \mathcal{C}_{\mathcal{J}_M}\} = \\ \widehat{\Gamma}_C(\{ X \widetilde{\mathcal{J}} X^{-1}  : X \in \mathcal{C}_{\mathcal{J}_M}\})$.

\smallskip

Moreover $\mathcal{SR}(M)_{(u_1, \cdots , u_p)}^{(\mu_{11}, \cdots ,\mu_{1s_1}; \cdots ; \mu_{r1}, \cdots ,\mu_{rs_r})}$ is a closed embedded submanifold of $GL_n$, diffeomorphic to the  homogeneous space $\dfrac{\mathcal{C}_{\mathcal{J}_M}}{\mathcal{C}_{\widetilde{{\mathcal{J}}}}}$.
\end{prop}
\begin{proof}
If $Y \in \mathcal{SR}(M)_{(u_1, \cdots , u_p)}^{(\mu_{11}, \cdots ,\mu_{1s_1}; \cdots ; \mu_{r1}, \cdots ,\mu_{rs_r})}$, then $Y =K \widetilde{\mathcal{J}}_Y K^{-1}$ for some $K \in GL_n$ and $Y^2=K (\widetilde{\mathcal{J}}_Y)^2 K^{-1} = K \mathcal{J}_M K^{-1} = C \mathcal{J}_M C^{-1} = M$. Hence $C^{-1}K \in \mathcal{C}_{{\mathcal{J}}_M}$ and, so, $K=CX$ for some $X \in \mathcal{C}_{{\mathcal{J}}_M}$. This gives one inclusion of the first equality. The reverse inclusion is a simple computation. The second equality follows directly from the definition of $\widehat{\Gamma}_C$.

Now $\widehat{\Gamma}_C$ is a diffeomorphism of $GL_n$, so it suffices to prove that the set 

$\{ X \widetilde{\mathcal{J}} X^{-1}  : X \in \mathcal{C}_{\mathcal{J}_M}\} = \mathcal{SR}(\mathcal{J}_M)_{(u_1, \cdots , u_p)}^{(\mu_{11}, \cdots ,\mu_{1s_1}; \cdots ; \mu_{r1}, \cdots ,\mu_{rs_r})}$ has the requested properties.
Let us consider the left action by conjugation of $\mathcal{C}_{\mathcal{J}_M}$ on $GL_n$. 

The set $\{ X \widetilde{\mathcal{J}} X^{-1}  : X \in \mathcal{C}_{\mathcal{J}_M}\}$ is the orbit of $\widetilde{\mathcal{J}}$. By Remark \ref{Popov}, this set is an immersed submanifold of $GL_n$, diffeomorphic to the  homogeneous space $\dfrac{\mathcal{C}_{\mathcal{J}_M}}{\mathcal{C}_{\widetilde{{\mathcal{J}}}}}$, being $\mathcal{C}_{\widetilde{{\mathcal{J}}}} \subseteq \mathcal{C}_{\mathcal{J}_M}$ and being $\mathcal{C}_{\widetilde{{\mathcal{J}}}}$ the isotropy subgroup of the action. 

Finally $\mathcal{SR}(\mathcal{J}_M)_{(u_1, \cdots , u_p)}^{(\mu_{11}, \cdots ,\mu_{1s_1}; \cdots ; \mu_{r1}, \cdots ,\mu_{rs_r})}$ is  closed in $GL_n$. Indeed if $\{Y_m\}$ is a sequence in $\mathcal{SR}(\mathcal{J}_M)_{(u_1, \cdots , u_p)}^{(\mu_{11}, \cdots ,\mu_{1s_1}; \cdots ; \mu_{r1}, \cdots ,\mu_{rs_r})}$, converging to $Y \in GL_n$, then 

$Y^2 =\mathcal{J}_M$ too and the characteristic polynomial of $Y$ is the same characteristic polynomial of all $Y_m$'s (constant with respect to $m$). 

Hence 
$Y \in \mathcal{SR}(\mathcal{J}_M)_{(u_1, \cdots , u_p)}^{(\mu_{11}, \cdots ,\mu_{1s_1}; \cdots ; \mu_{r1}, \cdots ,\mu_{rs_r})}$ and this last is closed and, therefore, it is an embedded submanifold of $GL_n$ (see for instance \cite[\S\,2.13  Theorem, p.\,65]{MonZip1055}).
\end{proof}

\begin{lemma}\label{lemma-su-J}

a) Let $\mathcal{J}$ be the matrix
\begin{center}
$\mathcal{J}= \big( \bigoplus_{i=1}^p \lambda_i I_{h_i}\big) \oplus \big(\bigoplus_{l=1}^r \bigoplus_{t=1}^{s_l} \rho_{lt} E_{\theta_l}^{\oplus m_{lt}}\big) \oplus \big( \bigoplus_{i=1}^{q}(- \zeta_i) I_{2 k_i} \big)
$,
\end{center}
where all indices satisfy the numerical conditions of Remarks-Definitions \ref{Jordan-form} (a).

Then the Lie group of non-singular matrices, commuting with $\mathcal{J}$, is 
\begin{center}
$\mathcal{C}_{\mathcal{J}} = (\bigoplus_{i=1}^p GL_{h_i}) \oplus (\bigoplus_{l=1}^r \bigoplus_{t=1}^{s_l} GL_{m_{lt}}(\mathbb{C})) \oplus (\bigoplus_{i=1}^q GL_{2k_i})$.
\end{center}

b) Let $\widetilde{\mathcal{J}}$be the matrix
$\widetilde{\mathcal{J}}:=$

$[\bigoplus_{i=1}^p (\sqrt{\lambda_i} (I_{u_i} \oplus (-I_{v_i})))]\oplus [\bigoplus_{l=1}^r \bigoplus_{t=1}^{s_l} \sqrt{\rho_{lt}} (E_{\theta_l/2}^{\oplus \mu_{lt}} \oplus E_{(\theta_l/2 -\pi)}^{\oplus \nu_{lt}} )] \oplus [\bigoplus_{i=1}^q \sqrt{\zeta_i} E_{\pi/2}^{\oplus k_i}]$,

where all indices satisfy the numerical conditions of Remarks-Definitions \ref{Jordan-form} (b).

Then the Lie group of non-singular matrices, commuting with $\widetilde{\mathcal{J}}$, is 

$\mathcal{C}_{\widetilde{\mathcal{J}}} = [\bigoplus_{i=1}^p (GL_{u_i} \oplus GL_{v_i})] \oplus [\bigoplus_{l=1}^r \bigoplus_{t=1}^{s_l}(GL_{\mu_{lt}}(\mathbb{C}) \oplus GL_{\nu_{lt}}(\mathbb{C}))] \oplus [\bigoplus_{i=1}^q GL_{k_i}(\mathbb{C})]$.
\end{lemma}

\begin{proof}
Part (a) is similar to \cite[Lemma\,3.3]{DoPe2020} and part (b) can be easily obtained by similar arguments.  
\end{proof}

\begin{thm}\label{SR(M)-metrica}
Let $M \in GL_n$ be a semisimple matrix, whose every (possible) negative eigenvalue has even multiplicity.
Then 

a) every connected component of $\mathcal{SR}(M)$ is a closed totally geodesic homogeneous semi-Riemannian submanifold of $(GL_n, g)$;

b) every manifold $\mathcal{SR}(M)_{(u_1, \cdots , u_p)}^{(\mu_{11}, \cdots ,\mu_{1s_1}; \cdots ; \mu_{r1}, \cdots ,\mu_{rs_r})}$, as in Remarks-Definitions \ref{Jordan-form} (c), is diffeomorphic to the following product of homogeneous spaces:
\begin{center}
$\big( \prod_{i=1}^p \dfrac{GL_{h_i}}{GL_{u_i} \oplus GL_{v_i}} \big) \times \big( \prod_{l=1}^r \prod_{t=1}^{s_l} \dfrac{GL_{m_{lt}}(\mathbb{C})}{GL_{\mu_{lt}}(\mathbb{C}) \oplus GL_{\nu_{lt}}(\mathbb{C})} \big) \times \big( \prod_{i=1}^q \dfrac{GL_{2 k_i}}{GL_{k_i}(\mathbb{C})}  \big)$,
\end{center}
where $v_i=h_i-u_i$ and $\nu_{lt}=  m_{lt}- \mu_{lt}$.
\end{thm}

\begin{proof}
Since $\mathcal{SR}(M) = \widehat{\Gamma}_C(\mathcal{SR}(\mathcal{J}_M))$ (with $C$ as in Proposition \ref{descr-SR(M)-indici}) and $\widehat{\Gamma}_C$ is an isometry of $(GL_n, g)$, it suffices to prove (a) for the matrix $\mathcal{J}_M$ instead of $M$. 
Every connected component of $\mathcal{SR}(\mathcal{J}_M)$ is contained in some 

$\mathcal{SR}(\mathcal{J}_M)_{(u_1, \cdots , u_p)}^{(\mu_{11}, \cdots ,\mu_{1s_1}; \cdots ; \mu_{r1}, \cdots ,\mu_{rs_r})} = \{ X \widetilde{\mathcal{J}} X^{-1}  : X \in \mathcal{C}_{\mathcal{J}_M}\}$ (remember the notations of Proposition \ref{descr-SR(M)-indici}) and this last is the orbit of $\widetilde{\mathcal{J}}$ under the action by conjugation of $\mathcal{C}_{\mathcal{J}_M}$ on $GL_n$. Since all conjugation maps are isometries of $(GL_n, g)$, every connected component of $\mathcal{SR}(\mathcal{J}_M)_{(u_1, \cdots , u_p)}^{(\mu_{11}, \cdots ,\mu_{1s_1}; \cdots ; \mu_{r1}, \cdots ,\mu_{rs_r})}$ is a semi-Riemannian submanifold of $(GL_n, g)$.

Now the map $L_M \circ j: (GL_n, g) \to (GL_n, g)$, $L_M \circ j(X)= M X^{-1}$ is an isometry of $(GL_n, g)$, whose set of fixed points is $\mathcal{SR}(M)$. Hence we can conclude (a) by \cite[Ex.\,8.1 p.\,61, Prop.\,8.3 p.\,56]{KoNo2}, remembering again Proposition \ref{descr-SR(M)-indici} too.

Part (b) follows from Proposition \ref{descr-SR(M)-indici}  and from Lemma \ref{lemma-su-J}.
\end{proof}

\begin{defi}\label{prin-gen}
Let $M \in GL_n$ be a semisimple matrix, whose every (possible) negative eigenvalue has even multiplicity. We say that a matrix $X$ is a \emph{generalized principal square root} of $M$, if $X^2=M$ and every eigenvalue of $X$ has principal argument in $[-\pi/2, \pi/2]$.

This definition is more general than the usual one of \emph{principal square root}  (see for instance \cite[Thm.\,1.29 p.\,20]{Hi2008}).

We denote by $\mathcal{PSR}(M)$ the set of all generalized principal square roots of $M$.
\end{defi}

\begin{rems}\label{princ} Let $M \in GL_n$ be a semisimple matrix, whose every (possible) negative eigenvalue has even multiplicity.

a) We have $\mathcal{PSR}(M) =\mathcal{SR}(M)_{(h_1, \cdots , h_p)}^{(m_{11}, \cdots ,m_{1s_1}; \cdots ; m_{r1}, \cdots , m_{rs_r})}$ and so, by Theorem \ref{SR(M)-metrica}, it is a single point if $M$ has no negative eigenvalue, according to the usual definition (see \cite[Thm.\,1.29 p.\,20]{Hi2008}), otherwise it is  diffeomorphic to the product of homogeneous spaces:
$\prod_{i=1}^q \dfrac{GL_{2 k_i}}{GL_{k_i}(\mathbb{C})}$ ($q$ is the number of distinct negative eigenvalues and $2k_1 , \cdots , 2k_q$ are their multiplicities).

In any case $\mathcal{PSR}(M)$ has $2^q$ connected components, since every factor has $2$ connected components.

b) Every $\mathcal{SR}(M)_{(u_1, \cdots , u_p)}^{(\mu_{11}, \cdots ,\mu_{1s_1}; \cdots ; \mu_{r1}, \cdots ,\mu_{rs_r})}$ has dimension 

$2[ \sum_{i=1}^p u_i(h_i- u_i)+ 2 \sum_{l=1}^r \sum_{t=1}^{s_l} m_{lt}(\mu_{lt}-\mu_{lt})  + \sum_{i=1}^q k_i^2 ]$.

From this, we get that $\mathcal{SR}(M)$ if finite if and only if the eigenvalues of $M$ have multiplicity $1$ and no eigenvalue of $M$ is negative. In this case $M$ has exactly $2^p \, 2^{(n-p)/2}=2^{(p+n)/2}$ distinct real square roots, where $p$ is the number of positive eigenvelues of $M$.
\end{rems}

\section{Symmetric case}\label{SSR}

Aim of this section is to study $\mathcal{SSR}(M)$: the set of all real symmetric square roots of $M$, with $M$ symmetric positive definite (Notations \ref{def-XSR}).

\begin{rem}
If $X \in GL_n$ is symmetric, then $X^2$ is symmetric positive definite. Hence, to have real symmetric square roots of a real symmetric matrix $M$, we must assume that $M$  is positive definite too.
\end{rem}

\begin{remsdefs}\label{Jordan-form-SSR}
Let $M \in GL_n$ be a symmetric positive definite matrix with distinct (real positive) eigenvalues $\lambda_1 < \lambda_2, \cdots  < \lambda_p$ with multiplicity $h_1, h_2, \cdots , h_p$ respectively, so the RJS form of $M$ is 
$\mathcal{J}_M = \bigoplus_{i=1}^p \lambda_i I_{h_i}$; in this case there exists $Q \in \mathcal{O}_n$ such that $M= Q \mathcal{J}_M Q^T$.

We denote by $\mathcal{SSR}(M)_{(u_1, \cdots , u_p)} = \mathcal{SR}(M)_{(u_1, \cdots , u_p)} \cap Sym_n$
(with $0 \le u_i \le h_i$ for every $1\le i \le p$): the set of all $Y \in \mathcal{SSR}(M)$, having  

$\widetilde{\mathcal{J}} = \bigoplus_{i=1}^p [\sqrt{\lambda_i} (I_{u_i} \oplus (-I_{(h_i-u_i)}))]$ as its RJA form.

We have $\mathcal{SSR}(M)_{(u_1, \cdots , u_p)} \subseteq GLSym_n(u)$ (with $u:=\sum_{i=1}^p u_i$) and, if 

$Y \in \mathcal{SSR}(M)_{(u_1, \cdots , u_p)}$, then there exists $R \in \mathcal{O}_n$ such that $Y = R \widetilde{\mathcal{J}} R^T$.

\smallskip

Note also that $\mathcal{SSR}(M) = \sqcup \, \mathcal{SSR}(M)_{(u_1, \cdots , u_p)}$, where the disjoint union is taken on all indices $u_1, \cdots , u_p$ as above. 
\end{remsdefs}

\begin{prop}\label{descr-SSR(M)-indici}
Let $M$ be a symmetric positive definite real matrix with distinct (real positive) eigenvalues $\lambda_1 < \lambda_2, \cdots  < \lambda_p$ with multiplicity $h_1, h_2, \cdots , h_p$ respectively and let $Q \in \mathcal{O}_n$ be such that $M= Q \mathcal{J}_M Q^T$. 

Fix any set $\mathcal{SSR}(M)_{(u_1, \cdots , u_p)}$ as above and denote by $\widetilde{\mathcal{J}}$ the RJA form of any $Y \in \mathcal{SSR}(M)_{(u_1, \cdots , u_p)}$. Then

\smallskip

$\mathcal{SSR}(M)_{(u_1, \cdots , u_p)} = \{Q X \widetilde{\mathcal{J}} X^T Q^T : X \in \mathcal{C}_{\mathcal{J}_M} \cap \mathcal{O}_n \} =  \Gamma_Q(\mathcal{SSR}(\mathcal{J}_M)_{(u_1, \cdots , u_p)})$,

where $\mathcal{SSR}(\mathcal{J}_M)_{(u_1, \cdots , u_p)}= \bigoplus_{i=1}^p \sqrt{\lambda_i} \, \big( \mathcal{O}_{h_i} \cap GLSym_{h_i}(u_i) \big)$.

\smallskip

Moreover $\mathcal{SSR}(M)_{(u_1, \cdots , u_p)}$ is a compact submanifold of $GLSym_n(u)$ 

(with $u:=\sum_{i=1}^p u_i$), diffeomorphic to the  homogeneous space $\dfrac{\mathcal{C}_{\mathcal{J}_M} \cap \mathcal{O}_n}{\mathcal{C}_{\widetilde{{\mathcal{J}}}} \cap \mathcal{O}_n}$.
\end{prop}

\begin{proof}
As in Propositions \ref{descr-SR(M)-indici} we can proof that 
$\mathcal{SSR}(M)_{(u_1, \cdots , u_p)} =  \\ \{Q X \widetilde{\mathcal{J}} X^T Q^T : X \in \mathcal{C}_{\mathcal{J}_M} \cap \mathcal{O}_n \} =  \Gamma_Q(\{ X \widetilde{\mathcal{J}}   X^T  : X \in  \mathcal{C}_{\mathcal{J}_M} \cap \mathcal{O}_n\})$

and $\{ X \widetilde{\mathcal{J}}   X^T  : X \in  \mathcal{C}_{\mathcal{J}_M} \cap \mathcal{O}_n \}$ is clearly equal to $\mathcal{SSR}(\mathcal{J}_M)_{(u_1, \cdots , u_p)}$.

By Lemma \ref{lemma-su-J} (a) we have 
$\mathcal{C}_{\mathcal{J}_M} \cap \mathcal{O}_n = (\bigoplus_{i=1}^p GL_{h_i}) \cap \mathcal{O}_n = \bigoplus_{i=1}^p \mathcal{O}_{h_i}$.

Hence, if $X = \oplus_{i=1}^p V_i  \in \bigoplus_{i=1}^p \mathcal{O}_{h_i}= \mathcal{C}_{\mathcal{J}_M} \cap \mathcal{O}_n$, then we obtain:

$X \widetilde{\mathcal{J}} X^T = \bigoplus_{i=1}^p \sqrt{\lambda_i} \, V_i (I_{u_i} \oplus (-I_{(h_i-u_i)}))V_{i}^T$.

In the last expression the $i$-th summand is clearly a generic element of 

$\sqrt{\lambda_i} \, (\mathcal{O}_{h_i} \cap GLSym_{h_i}(u_i))$; this concludes the first part.

Now $\mathcal{SSR}(M)_{(u_1, \cdots , u_p)}$ and $\mathcal{SSR}(\mathcal{J}_M)_{(u_1, \cdots , u_p)}$ are diffeomorphic, since $\Gamma_{Q}$ is a diffeomorphism of $GLSym_n(u)$.

By considering the action by congruence of $\mathcal{C}_{\mathcal{J}_M} \cap \mathcal{O}_n$ on $GLSym_n(u)$, 

$\mathcal{SSR}(\mathcal{J}_M)_{(u_1, \cdots , u_p)}$ is the orbit of $\widetilde{\mathcal{J}}$, while the isotropy subgroup at $\widetilde{\mathcal{J}}$ is 

$\mathcal{C}_{\widetilde{{\mathcal{J}}}} \cap \mathcal{O}_n \subseteq \mathcal{C}_{\mathcal{J}_M}  \cap \mathcal{O}_n$.

Hence $\mathcal{SSR}(\mathcal{J}_M)_{(u_1, \cdots , u_p)}$ (and also $\mathcal{SSR}(M)_{(u_1, \cdots , u_p)}$) is a compact submanifold of $GLSym_n(u)$ diffeomorphic to $\dfrac{\mathcal{C}_{\mathcal{J}_M} \cap \mathcal{O}_n}{\mathcal{C}_{\widetilde{{\mathcal{J}}}} \cap \mathcal{O}_n}$ (see Remark \ref{Popov}). This concludes the proof.
\end{proof}

\begin{thm}\label{SSR(M)-metrica}
Let $M$ be a symmetric positive definite real matrix with distinct (real positive) eigenvalues $\lambda_1 < \lambda_2 \cdots  < \lambda_p$ with multiplicity $h_1, h_2, \cdots , h_p$ respectively. Then

a) every manifold $\mathcal{SSR}(M)_{(u_1, \cdots , u_p)}$, as above, is diffeomorphic to the following product of homogeneous spaces:
\begin{center}
$\prod_{i=1}^p \dfrac{\mathcal{O}_{h_i}}{\mathcal{O}_{u_i} \oplus \mathcal{O}_{v_i}} $,
\end{center}
where $v_i=h_i-u_i$ (note that $\dfrac{\mathcal{O}_{h_i}}{\mathcal{O}_{u_i} \oplus \mathcal{O}_{v_i}}$ is diffeomorphic to the real Grassmannian $\mathbb{G}_{u_i}(\mathbb{R}^{h_i})$);

b) the connected components of $\mathcal{SSR}(M)$ are the manifolds $\mathcal{SSR}(M)_{(u_1, \cdots , u_p)}$ and each of them is a compact totally geodesic homogeneous semi-Riemannian submanifold of $(GL_n, g)$.
\end{thm}

\begin{proof} 
From Proposition \ref{descr-SSR(M)-indici}, $\mathcal{SSR}(M)_{(u_1, \cdots , u_p)}$ is diffeomorphic to the  homogeneous space $\dfrac{\mathcal{C}_{\mathcal{J}_M} \cap \mathcal{O}_n}{\mathcal{C}_{\widetilde{{\mathcal{J}}}} \cap \mathcal{O}_n}$. By Lemma \ref{lemma-su-J}, $\mathcal{C}_{\mathcal{J}_M} \cap \mathcal{O}_n = \bigoplus_{i=1}^p \mathcal{O}_{h_i}$ and 
$\mathcal{C}_{\widetilde{\mathcal{J}}} \cap \mathcal{O}_n = \bigoplus_{i=1}^p (\mathcal{O}_{u_i} \oplus \mathcal{O}_{v_i})$. This allows to conclude (a).

Every manifold $\dfrac{\mathcal{O}_{h_i}}{\mathcal{O}_{u_i} \oplus \mathcal{O}_{v_i}}$ is diffeomorphic to a real Grassmannian and, so, it is compact and connected. This implies that $\mathcal{SSR}(M)_{(u_1, \cdots , u_p)}$ is compact and connected too. Since $\mathcal{SSR}(M)$ is disjoint union of such manifolds, these last are its connected components.

As in the proof of Theorem \ref{SR(M)-metrica}, it is possible to prove that $\mathcal{SSR}(M)_{(u_1, \cdots , u_p)}$ is a semi-Riemannian submanifold of $(GL_n, g)$.

Now every connected component of $\mathcal{SR}(M)$ is a totally geodesic semi-Riemannian submanifold of $(GL_n, g)$ again by Theorem \ref{SR(M)-metrica} and, for every $0 \le q \le n$, $GLSym_n(q)$ is a totally geodesic semi-Riemannian submanifold of $(GL_n, g)$ by 

\cite[Prop.\,2.3]{DoPe2019}. 

We can conclude part (b), since $\mathcal{SSR}(M)_{(u_1, \cdots , u_p)}=\mathcal{SR}(M)_{(u_1, \cdots , u_p)} \cap GLSym_n(u)$ (with $u := \sum_{i=1}^p u_i$): intersection of two totally geodesic submanifolds.
\end{proof}

\begin{rems} Assume the same hypotheses and notations as in Theorem \ref{SSR(M)-metrica}.

a) If $u_1 = h_1, \cdots , u_p=h_p$ or if $u_1 = \cdots  = u_p =0$, then the manifold 

$\mathcal{SSR}(M)_{(u_1, \cdots , u_p)}$ is a single point. In the first case this point is a symmetric positive definite matrix: the usual (positive definite) square root of $M$, denoted by $\sqrt{M}=M^{1/2}$; in the second case it is a symmetric negative definite matrix and precisely $-\sqrt{M}=-M^{1/2}$.

b) $\mathcal{SSR}(M)$ is a finite set if and only if all eigenvalues of $M$ have multiplicity $1$; in this case $\mathcal{SSR}(M)$ has cardinality $2^n$ and $\mathcal{SR}(M)= \mathcal{SSR}(M)$.
\end{rems}

\section{Orthogonal case}\label{OSR}

Aim of this section is to study $\mathcal{OSR}(M)$, the set of all orthogonal square roots of $M$, where $M$ is in $S\mathcal{O}_n$ (Notations \ref{def-XSR}). We have already remarked that 

$Fix(L_M \circ j) = \mathcal{OSR}(M)$ (remember Remarks-Definitions \ref{def-isom} (b)).

\begin{remsdefs}\label{rem-RJS-RJA-ortog}

a) Remembering \ref{Jordan-form}, the RJS form of $M \in S\mathcal{O}_n$ is 
$$(*) \ \ \ \ \ \mathcal{J}_M= I_h \oplus E_{\theta_1}^{\oplus m_1} \oplus \cdots \oplus E_{\theta_r}^{\oplus  m_r} \oplus(-I_{2k})$$ 
with $h, r, k \ge 0$, $h+ 2 m_1 + \cdots + 2 m_r + 2k =n$ and $0 < \theta_1 < \theta_2 < \cdots < \theta_r < \pi$;

the eigenvalues of $M$ are: $1$ with multiplicity $h \ge 0$, $\exp(\pm {\bf i} \theta_1)$ both with multiplicity $m_1$, $\cdots$ , up to $\exp(\pm {\bf i} \theta_r)$ both with multiplicity $m_r$ ($m_j >0 \ \ \forall j$, if $r > 0$) and $-1$ with multiplicity $2 k \ge 0$.

The eigenvalues of any (orthogonal) square root $Y$ of $M$ are necessarily: $1$ with multiplicity $u$, $-1$ with multiplicity $h-u$ (where $0 \le u \le h$),  $\exp(\pm  \frac{1}{2} {\bf i} \theta_j)$ both with multiplicity $\mu_j$ and $\exp(\pm {\bf i}( \frac{1}{2} \theta_j - \pi ))$ both with multiplicity $m_j - \mu_j$ ($0 \le \mu_j \le m_j$), for every $1 \le j \le r$, if $r >0$, and finally $\pm {\bf i}$ both with multiplicity  $k$, so that the RJA form of $Y$ is the following:
$$(**) \ \ \ \ \ 
\widetilde{\mathcal{J}} := I_u \oplus (- I_{(h-u)}) \oplus E_{\frac{1}{2} \theta_1}^{\oplus \mu_1} \oplus E_{(\frac{1}{2} \theta_1 - \pi)}^{\oplus (m_1-\mu_1)} \oplus \cdots \oplus E_{\frac{1}{2} \theta_r}^{\oplus \mu_r} \oplus E_{(\frac{1}{2} \theta_r - \pi)}^{\oplus (m_r-\mu_r)} \oplus E_{\frac{\pi}{2}}^{\oplus k}. 
$$

b) We denote by $\mathcal{OSR}(M)_u^{(\mu_1, \cdots , \mu_r)}$ the subset of $\mathcal{OSR}(M)$ of matrices whose RJA form is of the type $(**)$ above.  

If $Y \in \mathcal{OSR}(M)_u^{(\mu_1, \cdots , \mu_r)}$, then there is a matrix $H \in \mathcal{O}_n$ such that $Y = H \widetilde{\mathcal{J}} H^T$.

On the indices $u$ and $\mu_1, \cdots , \mu_r$ we can make analogous observations as in Remarks-Definitions \ref{Jordan-form} (c); for instance, the matrix $M$, the index $u$ and the family of indices $(\mu_1, \cdots , \mu_r$) determine the RJA form above; moreover the same agreements as in Remarks-Definitions \ref{Jordan-form} (c) hold, if one or both families of indices are empty.

Note that, if both families of indices are empty (i.e. $h=r=0$), then $M=- I_n$ with $n$ even.

Of course $\det(Y)= (-1)^{h-u}$ for every $Y \in \mathcal{OSR}(M)_u^{(\mu_1, \cdots , \mu_r)}$ and, if $h >0$ or $r> 0$, then $\mathcal{OSR}(M)$ is the disjoint union of all $\mathcal{OSR}(M)_u^{(\mu_1, \cdots , \mu_r)}$ 
for every $0 \le u \le h$ (if $h>0$) and $0 \le \mu_j \le m_j$, $1 \le j \le r$ (if $r>0$).

In the following statements of this section it is understood that one or both families of the above indices can  be empty.
\end{remsdefs}

\begin{prop}\label{prop-ortog}
Let $M \in S \mathcal{O}_n$ whose RJS form, $\mathcal{J}_M$, is as in Remarks-Definitions \ref{rem-RJS-RJA-ortog} $(*)$ and let $C_0 \in \mathcal{O}_n$ such that $M= C_0 \mathcal{J}_M C_0^T$. 
For every $0 \le u \le h$, $0 \le \mu_j \le m_j$, $1 \le j \le r$, let 
$\widetilde{\mathcal{J}}$ be the matrix of the form $(**)$ as in Remarks-Definitions \ref{rem-RJS-RJA-ortog}. Then

\smallskip

$
\mathcal{OSR}(M)_u^{(\mu_1, \cdots , \mu_r)}= \{ C_0 X \widetilde{\mathcal{J}}   X^T C_0^T : X \in  \mathcal{C}_{\mathcal{J}_M} \cap \mathcal{O}_n\} =  \Gamma_{C_0}(\mathcal{OSR}(\mathcal{J}_M)_u^{(\mu_1, \cdots , \mu_r)}),
$

\smallskip

where 
$\mathcal{OSR}(\mathcal{J}_M)_u^{(\mu_1, \cdots , \mu_r)}:=$

$[\mathcal{O}_h \cap Sym_h(u)] \oplus [ \bigoplus_{j=1}^r \exp(\dfrac{1}{2}{\bf i} \theta_j) (U_{m_j} \cap Herm_{m_j}(\mu_j) ]\oplus [\mathcal{O}_{2k} \cap \mathfrak{so}_{2k}]$.

Moreover $\mathcal{OSR}(M)_u^{(\mu_1, \cdots , \mu_r)}$ is a compact submanifold of $\mathcal{O}_n$, diffeomorphic to the homogeneous space $\dfrac{\mathcal{C}_{\mathcal{J}_M} \cap \mathcal{O}_n}{\mathcal{C}_{\widetilde{{\mathcal{J}}}} \cap \mathcal{O}_n}$.
\end{prop}

\begin{proof}
As in Propositions \ref{descr-SR(M)-indici} and \ref{descr-SSR(M)-indici}, we have: 

$
\mathcal{OSR}(M)_u^{(\mu_1, \cdots , \mu_r)}= \{ C_0 X \widetilde{\mathcal{J}}   X^T C_0^T : X \in  \mathcal{C}_{\mathcal{J}_M} \cap \mathcal{O}_n\} = \\ \Gamma_{C_0}(\{ X \widetilde{\mathcal{J}}   X^T  : X \in  \mathcal{C}_{\mathcal{J}_M} \cap \mathcal{O}_n\})$

and $\{ X \widetilde{\mathcal{J}}   X^T  : X \in  \mathcal{C}_{\mathcal{J}_M} \cap \mathcal{O}_n\} = \mathcal{OSR}(\mathcal{J}_M)_u^{(\mu_1, \cdots , \mu_r)}$.

Analogously to Proposition \ref{descr-SSR(M)-indici}, by Lemma \ref{lemma-su-J} (a), we get: 
$\mathcal{C}_{\mathcal{J}_M} \cap \mathcal{O}_n = \\ \big(GL_h \oplus GL_{m_1}(\mathbb{C}) \oplus \cdots \oplus  GL_{m_r}(\mathbb{C}) \oplus GL_{2k} \big) \cap \mathcal{O}_n =  
\mathcal{O}_h \oplus U_{m_1} \oplus \cdots \oplus U_{m_r} \oplus \mathcal{O}_{2k}
$.

Hence, if $X = V \oplus Z_1 \oplus \cdots \oplus Z_r \oplus W \in \mathcal{O}_h \oplus U_{m_1} \oplus \cdots \oplus U_{m_r} \oplus \mathcal{O}_{2k}= \mathcal{C}_{\mathcal{J}_M} \cap \mathcal{O}_n$, then, remembering Remark-Definition \ref{rho}, we obtain: 
$X \widetilde{\mathcal{J}} X^T =$ 

$[V (I_u \oplus (-I_{(h-u)}))V^T] \oplus [\bigoplus_{j=1}^r exp(\dfrac{1}{2}{\bf i}\theta_j) Z_j [I_{\mu_j} \oplus (- I_{(m_j-\mu_j)})] Z_j^*] \oplus [W (E_{\frac{\pi}{2}}^{\oplus k}) W^T]$.

In the last expression the first and the last summand are clearly generic elements of $\mathcal{O}_h \cap Sym_h(u)$ and of $\mathcal{O}_{2k} \cap \mathfrak{so}_{2k}$, respectively. Finally, for every $j$, the matrix 
$Z_j [I_{\mu_j} \oplus (- I_{(m_j-\mu_j)})] Z_j^*$ is a generic element of $U_{m_j} \cap Herm_{m_j}(\mu_j)$; this concludes the first part.

Again $\mathcal{OSR}(M)_u^{(\mu_1, \cdots , \mu_r)}$ and $\mathcal{OSR}({\mathcal{J}_M})_u^{(\mu_1, \cdots , \mu_r)}$ are diffeomorphic, since $\Gamma_{C_0}$ is a diffeomorphism of $\mathcal{O}_n$.

In the action by congruence of $\mathcal{C}_{\mathcal{J}_M} \cap \mathcal{O}_n$ on $\mathcal{O}_n$, $\mathcal{OSR}({\mathcal{J}_M})_u^{(\mu_1, \cdots , \mu_r)}$ is the orbit of $\widetilde{\mathcal{J}}$, while the isotropy subgroup at $\widetilde{\mathcal{J}}$ is $\mathcal{C}_{\widetilde{{\mathcal{J}}}} \cap \mathcal{O}_n \subseteq \mathcal{C}_{\mathcal{J}_M} \cap \mathcal{O}_n$; so $\mathcal{OSR}({\mathcal{J}_M})_u^{(\mu_1, \cdots , \mu_r)}$ and  $\mathcal{OSR}(M)_u^{(\mu_1, \cdots , \mu_r)}$ are  compact submanifolds of $\mathcal{O}_n$ diffeomorphic to $\dfrac{\mathcal{C}_{\mathcal{J}_M} \cap \mathcal{O}_n}{\mathcal{C}_{\widetilde{{\mathcal{J}}}} \cap \mathcal{O}_n}$ (again by Remark \ref{Popov}). This concludes the proof.
\end{proof}

\begin{thm}\label{manifold-O-con-indici}
Let $M \in S \mathcal{O}_n$ whose RJS form, $\mathcal{J}_M$, is as in Remarks-Definitions \ref{rem-RJS-RJA-ortog} $(*)$ and let $C_0 \in \mathcal{O}_n$ such that $M= C_0 \mathcal{J}_M C_0^T$. 
For every $0 \le u \le h$, $0 \le \mu_j \le m_j$, $1 \le j \le r$, we have:

\smallskip

a) $\mathcal{OSR}(M)_u^{(\mu_1, \cdots , \mu_r)}$ is a compact homogeneous differentiable manifold  diffeomorphic to the product 
$$\dfrac{\mathcal{O}_h}{\mathcal{O}_u \oplus \mathcal{O}_{(h-u)}} \times \prod_{j=1}^r \dfrac{U_{m_j}}{U_{\mu_j} \oplus U_{(m_j-\mu_j)}} \times \dfrac{\mathcal{O}_{2k}}{U_k}$$
whose dimension is $u(h-u) + 2 \sum_{j=1}^r \mu_j(m_j-\mu_j) +k(k-1)$.

If $k=0$ (i.e if $-1$ is not an eigenvalue of $M$), the manifold $\mathcal{OSR}(M)_u^{(\mu_1, \cdots , \mu_r)}$ is connected, otherwise it has two connected components and both are diffeomorphic to 
$\dfrac{\mathcal{O}_h}{\mathcal{O}_u \oplus \mathcal{O}_{(h-u)}} \times \prod_{j=1}^r \dfrac{U_{m_j}}{U_{\mu_j} \oplus U_{(m_j-\mu_j)}} \times \dfrac{S\mathcal{O}_{2k}}{U_k}$.

\smallskip

b) $\mathcal{OSR}(M)_u^{(\mu_1, \cdots , \mu_r)}$ is also a totally geodesic homogeneous Riemannian submanifold of $(\mathcal{O}_n, - g)$. 
\end{thm}

\begin{proof}
In the proof of  Proposition \ref{prop-ortog} we have seen that

$\mathcal{C}_{\mathcal{J}_M} \cap \mathcal{O}_n = 
\mathcal{O}_h \oplus U_{m_1} \oplus \cdots \oplus U_{m_r} \oplus \mathcal{O}_{2k}
$.

By Lemma \ref{lemma-su-J} (b), we get also: 
$\mathcal{C}_{\widetilde{\mathcal{J}}} \cap \mathcal{O}_n=\\
 \{ \big(GL_u \oplus GL_{(h-u)} \big) \oplus \big( \bigoplus_{j=1}^r GL_{\mu_j}(\mathbb{C}) \oplus GL_{(m_j - \mu_j)}(\mathbb{C}) \big) \oplus GL_{k}(\mathbb{C}) \} \cap \mathcal{O}_n  =\\
\mathcal{O}_u \oplus \mathcal{O}_{(h-u)} \oplus U_{\mu_1} \oplus U_{(m_1 - \mu_1)} \oplus \cdots \oplus U_{\mu_r} \oplus U_{(m_r - \mu_r)} \oplus U_{k}
$.

Taking into account Proposition \ref{prop-ortog}, we obtain the first part of (a), because it is easy to see that the corresponding quotient is naturally diffeomorphic to the product in the statement; the dimension is a trivial computation on the product of the quotients. Since every factor is compact, the product is compact too. The factor $\dfrac{\mathcal{O}_h}{\mathcal{O}_u \oplus \mathcal{O}_{(h-u)}}$ is diffeomorphic to the real Grassmannian 
$\mathbb{G}_u(\mathbb{R}^h)$, every factor $\dfrac{U_{m_j}}{U_{\mu_j} \oplus U_{(m_j-\mu_j)}}$ is diffeomorphic to the complex Grassmannian 
$\mathbb{G}_{\mu_j}(\mathbb{C}^{m_j})$, while the last factor  $\dfrac{\mathcal{O}_{2k}}{U_k}$ is diffeomorphic to the manifold of skew-symmetric orthogonal matrices of order $2k$. This allows to obtain the statement about the connectedness (and hence to conclude (a)), for instance, by means of the description of $\dfrac{\mathcal{O}_{2k}}{U_k}$, given in \cite[Rem.\,3.1 and Prop.\,3.2]{DoPe2018a}).

\smallskip

As in Theorem \ref{SR(M)-metrica}, the map $L_M \circ j: (\mathcal{O}_n, - g) \to (\mathcal{O}_n, - g)$, $L_M \circ j(X)= M X^{-1}$ is an isometry of $(\mathcal{O}_n, - g)$, whose set of fixed points is $\mathcal{OSR}(M)$; hence we can conclude (b) by \cite[Thm.\,5.1, pp.\,59-60]{Koba1995}.
\end{proof}

\begin{rems}\label{OSR-casi}  

a) We have: $\mathcal{OSR}(-I_{2m}) = \mathcal{O}_{2m} \cap \mathfrak{so}_{2m} = S\mathcal{O}_{2m} \cap \mathfrak{so}_{2m}$ for every $m \ge 1$
 and $(S\mathcal{O}_{2m} \cap \mathfrak{so}_{2m}, - g)$ is a compact totally geodesic Riemannian submanifold of $(S\mathcal{O}_{2m}, - g)$ with two connected components identified by the sign of the \emph{pfaffian} of their matrices. Both components are simply connected and diffeomorphic to $\dfrac{S\mathcal{O}_{2m}}{U_m}$.

This follows from Proposition \ref{prop-ortog} and Theorem \ref{manifold-O-con-indici}, taking into account the analysis developed in \cite[\S\,3]{DoPe2018a}.

\smallskip

b) We have:  $\mathcal{OSR}({I_n})_u = \mathcal{O}_n \cap GLSym_n(u)$ for every $0 \le u \le n$ and
this is a compact totally geodesic Riemannian submanifold of $(\mathcal{O}_n, - g)$, diffeomorphic to the real Grassmannian $\mathbb{G}_{u}(\mathbb{R}^{n})$.

This follows again from Proposition \ref{prop-ortog} and Theorem \ref{manifold-O-con-indici} (see also 

\cite[\S\,4]{DoPe2018a}).
\end{rems}

\begin{rem}
Let $M \in S\mathcal{O}_n$ and assume the same notations as in Remarks-Definitions \ref{rem-RJS-RJA-ortog}. Then
$\mathcal{OSR}(M)$ is the disjoint union of the sets  $\mathcal{OSR}({M})_u^{(\mu_1, \cdots , \mu_r)}$ (on all admissible indices $u, \mu_1, \cdots , \mu_r$) and its connected components are homogeneous differentiable manifolds of various dimensions. The number of such components can be obtained from Theorem \ref{manifold-O-con-indici}.

In particular there are finitely many orthogonal square roots of $M$ if and only if each possible $\mathcal{OSR}({M})_u^{(\mu_1, \cdots , \mu_r)}$ has dimension $0$, \  i.\,e. if and only if 

$0 \le h, m_1, \cdots m_r, k \le 1$, \ i.\,e. if and only if $M$ has the eigenvalues different from $-1$, all of multiplicity $1$, and, if $-1$ is an eigenvalue, its multiplicity is $2$.
Hence, if there are finitely many orthogonal square roots of $M$, their number is precisely $2^{\lfloor (n+1)/2 \rfloor}$, where $\lfloor (n+1)/2 \rfloor$ denotes the integer part of $\dfrac{n+1}{2}$.
\end{rem}

\begin{rem}\label{OSR-intersecato}
Let $M \in S\mathcal{O}_n$ and assume the same notations as in Remarks-Definitions \ref{rem-RJS-RJA-ortog}.

If $1$ is not an eigenvalue of $M$ (i.e. if $h=0$), then $\mathcal{OSR}({M}) \subseteq S\mathcal{O}_n$.

Otherwise (i.e. if $h \ge 1$) $\mathcal{OSR}({M}) \cap S\mathcal{O}_n$ and $\mathcal{OSR}({M}) \cap \mathcal{O}_n^-$ are both not empty with

$\mathcal{OSR}({M}) \cap S\mathcal{O}_n = \bigsqcup_{\ 0 \le i \le \lfloor h/2 \rfloor; \ 0 \le \mu_j \le m_j ; \  1 \le j \le r \ } \mathcal{OSR}({M})_{h-2i}^{(\mu_1, \cdots , \mu_r)}$,

$\mathcal{OSR}({M}) \cap \mathcal{O}_n^- = \bigsqcup_{\ 0 \le i \le \lfloor (h-1)/2 \rfloor; \ 0 \le \mu_j \le m_j ; \  1 \le j \le r \ } \mathcal{OSR}({M})_{h-2i-1}^{(\mu_1, \cdots , \mu_r)}$.

Indeed, by Remarks-Definitions \ref{rem-RJS-RJA-ortog}, $\det(Y) = (-1)^{h-u}$, if $Y \in \mathcal{OSR}(M)_u^{(\mu_1, \cdots , \mu_r)}$.
\end{rem}

\begin{rem} Let $M \in S\mathcal{O}_n$ and, analogously to Definition \ref{prin-gen} and to Remarks \ref{princ}, we can define the set $\mathcal{POSR}(M)$ of \emph{generalized principal orthogonal square roots} of $M$ and, with the same notations as in Remarks-Definitions \ref{rem-RJS-RJA-ortog}, we get:
$\mathcal{POSR}(M) = \mathcal{OSR}(M)_h^{(m_1, \cdots , m_r)}$.

By Theorem \ref{manifold-O-con-indici}, $\mathcal{POSR}(M)$ is a single point if $-1$ is not an eingenvalue of $M$, while it is diffeomorphic to $\dfrac{\mathcal{O}_{2k}}{U_k}$ \  if $-1$ has multiplicity $2k \ge 2$ as eingenvalue of $M$.

In particular, if $k=1$, $\mathcal{POSR}(M)$ consists of two distinct points.

We refer, for instance, to \cite[\S\,3]{DoPe2018a} for more properties of the homogeneous space $\dfrac{\mathcal{O}_{2k}}{U_k}$.
\end{rem}


\begin{thebibliography}{}

\bibitem[Amari 2016]{Amari2016} AMARI Shun-ichi, \emph{Information Geometry and its Applications}, Applied Mathematical Sciences, vol. 194, Springer, Berlin.

\smallskip

\bibitem[Barbaresco-Nielsen 2017]{BarNie2017} BARBARESCO Fr\'ed\'eric, NIELSEN Frank, (eds.), \emph{Differential Geometrical Theory of Statistics}, [Special Issue], Entropy - MDPI, Basel.

\smallskip

\bibitem[Bhatia 2007]{Bha2007}  BHATIA Rajendra, \emph{Positive Definite Matrices}, Princeton University Press, Princeton.

\smallskip

\bibitem[Bhatia 2013]{Bha2013} BHATIA Rajendra, ``The Riemannian Mean of Positive Matrices'', in Nielsen, F., Bhatia, R. (eds.): \emph{Matrix Information Geometry}, Springer, Berlin, 35--51.

\smallskip

\bibitem[Bhatia-Holbrook 2006]{BhaH2006} BHATIA Rajendra, HOLBROOK John,  ``Riemannian geometry and matrix geometric means'', in \emph{Linear Algebra Appl.} 413, 594--618.


\smallskip

\bibitem[Bridson-Haefliger 1999]{BridHaef1999} BRIDSON Martin R., HAEFLIGER Andr\'e, \emph{Metric Spaces of Non-Positive Curvature}, GMW 319, Springer-Verlag, Berlin.

\smallskip

\bibitem[Dolcetti-Pertici 2015]{DoPe2015} DOLCETTI Alberto, PERTICI Donato, ``Some differential properties of $GL_n(\mathbb{R})$ with the trace metric'', \emph{Riv. Mat. Univ. Parma}, Vol. 6, No. 2, 267--286.

\smallskip

\bibitem[Dolcetti-Pertici 2017]{DoPe2017} DOLCETTI Alberto, PERTICI Donato, ``Some remarks on the Jordan-Chevalley decomposition'', S\~{a}o Paulo J. Math. Sci, 11(2): 385--404.

\smallskip

\bibitem[Dolcetti-Pertici 2018a]{DoPe2018a} DOLCETTI Alberto, PERTICI Donato, ``Skew symmetric logarithms and geodesics on $O_n(\mathbb{R})$",  \emph{Adv. Geom.}; 18(4): 495--507.

\smallskip

\bibitem[Dolcetti-Pertici 2018b]{DoPe2018b} DOLCETTI Alberto, PERTICI Donato, ``Some additive decompositions
of semisimple matrices'', \emph{Rend. Istit. Mat. Univ. Trieste}
Volume 50, 47--63.

\smallskip

\bibitem[Dolcetti-Pertici 2019]{DoPe2019} DOLCETTI Alberto, PERTICI Donato, ``Differential properties of spaces of symmetric real matrices'', \emph{Rendiconti Sem. Mat. Univ. Pol. Torino}
Vol. 77, 1 (2019), 25--43.

\smallskip

\bibitem[Dolcetti-Pertici 2020]{DoPe2020} DOLCETTI Alberto, PERTICI Donato, ``Elliptic isometries of the manifold of positive definite real matrices with the trace metric'', \emph{Rend. Circ. Mat. Palermo II. Ser}, \url{https://doi.org/10.1007/s12215-020-00510-9}.

\smallskip


\bibitem[Higham 2008]{Hi2008} HIGHAM Nicholas J., \emph{Functions of Matrices. Theory and Computation}, SIAM Society for Industrial and Applied Mathematics, Phildelphia.

\smallskip

\bibitem[Higham 2020]{Hi2020} HIGHAM Nicholas J., \emph{What is a Matrix Square Root?}, viewed 10 August 2020, $<$\url{https://nhigham.com/2020/05/21/what-is-a-matrix-square-root/}$>$.

\smallskip

\bibitem[Horn-Johnson 1991]{HoJ1991}HORN Roger A.,  JOHNSON Charles R., \emph{Topics in Matrix Analysis}, Cambridge University Press, Cambridge.

\smallskip

\bibitem[Horn-Johnson 2013]{HoJ2013} HORN Roger A.,  JOHNSON Charles R., \emph{Matrix analysis}, Second Edition, Cambridge University Press, Cambridge.

\smallskip

\bibitem[Humphreys 1975]{Humph1975} HUMPHREYS, James E., \emph{Linear Algebraic Groups}, GTM 21. Springer, New York.

\smallskip

\bibitem[Kobayashi 1995]{Koba1995}, KOBAYASHI Shoshichi, \emph{Transformation {G}roups in {D}ifferential {G}eometry}, Springer-Verlag, Berlin.

\smallskip

\bibitem[Kobayashi-Nomizu 1969]{KoNo2} KOBAYASHI Shoshichi, NOMIZU Katsumi, \emph{Foundations of differential geometry}, Vol. 2, Interscience Publishers John Wiley and Sons, New York.

\smallskip

\bibitem[Montgomery-Zippin 1955]{MonZip1055} MONTGOMERY Deane, ZIPPIN Leo, \emph{Topological Transformation Groups}, Interscience Tracts in Pure and Applied Mathematics, New York.

\smallskip

\bibitem[Moakher-Z\'era\"i 2011]{MoZ2011} MOAKHER Maher, Z\'ERA\"I Mourad, ``The Riemannian Geometry of The Space of Positive-Definite Matrices and Its Application to the Regularization of Positive-Definite Matrix-Valued Data'', in \emph{J. Math. Imaging Vis.}, 40, no. 2, 171--187.

\smallskip

\bibitem[Nielsen-Barbaresco 2019]{NieBar2019} NIELSEN Frank, BARBARESCO Fr\'ed\'eric (eds.), \emph{Geometric Science of Information}, GSI 2019, Lecture Notes in Computer Science, Vol. 11712, Springer, Cham.

\smallskip


\bibitem[Popov 1991]{Popov1991} POPOV Valentin L., Orbit, in \emph{Encyclopedia of Mathematics}, Vol. 7, Kluver Academic Press Publishers, Dordrecht.

\bibitem[Terras 2016]{Terras2016} TERRAS Audrey, \emph{Harmonic Analysis on
Symmetric Spaces - Higher Rank Spaces, Positive Definite Matrix Space and Generalizations}, Second Edition, Springer, New York.
\end{thebibliography}
\end{document}